\documentclass[12pt,a4paper]{amsart}%

\usepackage{amsfonts,amsmath,amssymb}
\usepackage{amssymb,amsthm,amsxtra}
\usepackage[usenames]{color}
\usepackage{amscd}
\usepackage{amsthm}
\usepackage{amsfonts}
\usepackage{amssymb}
\usepackage{amsmath}
\usepackage{graphicx}
\usepackage{hyperref}
\usepackage{enumerate}%
\usepackage[all]{xy}
\usepackage[usenames,dvipsnames]{xcolor}
\usepackage{mathrsfs}
\usepackage[margin=1.45in]{geometry}
\usepackage{cleveref}

 \newtheorem*{corollary*}{Corollary}
 \newtheorem*{construction*}{Construction}
 \newtheorem*{definition*}{Definition}
 \newtheorem*{notation*}{Notation}
 \newtheorem*{lemma*}{Lemma}
 \newtheorem*{theorem*}{Theorem}
 \newtheorem*{remark*}{Remark}
 \newtheorem*{example*}{Example}
 \newtheorem*{conjecture*}{Conjecture}
 \newtheorem*{condition*}{Condition}
 \newtheorem*{result*}{Result}
 \newtheorem*{property*}{Property}

 \newtheorem*{cor*}{Corollary}
 \newtheorem*{const*}{Construction}
 \newtheorem*{defn*}{Definition}
 \newtheorem*{notn*}{Notation}
 \newtheorem*{lem*}{Lemma}
 \newtheorem*{thm*}{Theorem}
 \newtheorem*{rem*}{Remark}
 \newtheorem*{exm*}{Example}
 \newtheorem*{conj*}{Conjecture}

 \newtheorem{lemma}{Lemma}[section]
 \newtheorem{proposition}[lemma]{Proposition}
 \newtheorem{remark}[lemma]{Remark}
 
 \newtheorem{theorem}[lemma]{Theorem}

 \newtheorem{thm}[lemma]{Theorem}
 \newtheorem{prop}[lemma]{Proposition}
 
 \newtheorem{defn}[lemma]{Definition}
 \newtheorem{notn}[lemma]{Notation}
 \newtheorem{cor}[lemma]{Corollary}

 \newtheorem{introtheorem}{Theorem}
 
 \crefname{introtheorem}{theorem}{theorems}
 \Crefname{introtheorem}{Theorem}{Theorems}
  \newtheorem{introthm}[introtheorem]{Theorem}
   \crefname{introthm}{theorem}{theorems}
 \Crefname{introthm}{Theorem}{Theorems}

  \crefname{introcorollary}{corollary}{corollaries}
 \Crefname{introcorollary}{Corollary}{Corollaries}

 \newtheorem{introcor}[introtheorem]{Corollary}
   \crefname{introcor}{corollary}{corollaries}
 \Crefname{introcor}{Corollary}{Corollaries}
 
   \crefname{introconjecture}{conjectures}{conjectures}
 \Crefname{introconjecture}{Conjecture}{Conjectures}
 
    \crefname{introconj}{conjectures}{conjectures}
 \Crefname{introconj}{Conjecture}{Conjectures}

     \crefname{introlem}{lemma}{lemmas}
 \Crefname{introlem}{Lemma}{Lemmas}
 
 \crefname{introremark}{remark}{remarks}
 \Crefname{introremark}{Remark}{Remarks}
 
  \crefname{introrem}{remark}{remarks}
 \Crefname{introrem}{Remark}{Remarks}

 \newtheorem{introprop}[introtheorem]{Proposition}
   \crefname{introprop}{Proposition}{Propositions}
 \Crefname{introprop}{Proposition}{Propositions}

   \crefname{introdefn}{definition}{definitions}
 \Crefname{introdefn}{Definition}{Definitions}
 
   \crefname{intronotn}{notation}{notations}
 \Crefname{intronotn}{Notation}{Notations}
 
   \crefname{introtask}{task}{tasks}
 \Crefname{introtask}{Task}{Tasks}
 
  \crefname{introprob}{problem}{problems}
 \Crefname{introprob}{Problem}{Problems}
 
   \crefname{introquestion}{question}{questions}
 \Crefname{introquestion}{Question}{Questions}

 \crefname{theorem}{theorem}{theorems}
 \Crefname{theorem}{Theorem}{Theorems}
  \crefname{thm}{theorem}{theorems}
 \Crefname{thm}{Theorem}{Theorems}

  \crefname{corollary}{Corollary}{Corollaries}
 \Crefname{corollary}{Corollary}{Corollaries}

   \crefname{cor}{Corollary}{Corollaries}
 \Crefname{cor}{Corollary}{Corollaries}

   \crefname{conjecture}{conjectures}{conjectures}
 \Crefname{conjecture}{Conjecture}{Conjectures}

    \crefname{conj}{conjectures}{conjectures}
 \Crefname{conj}{Conjecture}{Conjectures}

     \crefname{lem}{lemma}{lemmas}
 \Crefname{lem}{Lemma}{Lemmas}

      \crefname{lemma}{Lemma}{Lemmas}
 \Crefname{lemma}{Lemma}{Lemmas}

 \crefname{remark}{remark}{remarks}
 \Crefname{remark}{Remark}{Remarks}

  \crefname{rem}{remark}{remarks}
 \Crefname{rem}{Remark}{Remarks}

   \crefname{rem}{remark}{remarks}
 \Crefname{rem}{Remark}{Remarks}

   \crefname{proposition}{Proposition}{Proposition}
 \Crefname{proposition}{Proposition}{Proposition}

    \crefname{prop}{Proposition}{Propositions}
 \Crefname{prop}{Proposition}{Propositions}

   \crefname{defn}{definition}{definitions}
 \Crefname{defn}{Definition}{Definitions}

   \crefname{notn}{notation}{notations}
 \Crefname{notn}{Notation}{Notations}

   \crefname{task}{task}{tasks}
 \Crefname{task}{Task}{Tasks}

  \crefname{prob}{problem}{problems}
 \Crefname{prob}{Problem}{Problems}

   \crefname{question}{question}{questions}
 \Crefname{question}{Question}{Questions}

\newcommand{\alp}{\alpha}

\renewcommand{\Im}{\operatorname{Im}}

\newcommand{\Exp}{\operatorname{Exp}}

\newcommand{\GL}{\operatorname{GL}}

\newcommand{\sll}{{\mathfrak{sl}}}

\newcommand{\oH}{\operatorname{H}}

\newcommand{\Ad}{\operatorname{Ad}}

\newcommand{\WF}{\operatorname{WF}}
\newcommand{\Irr}{\operatorname{Irr}}

\newcommand{\WO}{\operatorname{WO}}

\newcommand{\C}{\mathbb{C}}

\newcommand{\bfG}{\mathbf{G}}
\newcommand{\bfH}{\mathbf{H}}

\newcommand{\R}{\mathbb{R}}

\newcommand{\Sc}{\cS}

\newcommand{\Rep}{\operatorname{Rep}}

\newcommand{\Fre}{Fr\'echet }

\newcommand{\onto}{\twoheadrightarrow}

\providecommand{\fg}{\mathfrak{g}}
\providecommand{\fh}{\mathfrak{h}}

\providecommand{\fl}{\mathfrak{l}}

\providecommand{\fn}{\mathfrak{n}}
\providecommand{\fp}{\mathfrak{p}}
\providecommand{\fq}{\mathfrak{q}}

\providecommand{\fs}{\mathfrak{s}}

\providecommand{\fu}{\mathfrak{u}}
\providecommand{\fv}{\mathfrak{v}}

\providecommand{\cM}{\mathcal{M}}
\providecommand{\cN}{\mathcal{N}}
\providecommand{\cO}{\mathcal{O}}

\providecommand{\cS}{\mathcal{S}}

\newcommand{\simpAr}[2][r]{%
\ar@{}[#1]|-*[@]_{#2}%
}

\newcommand{\DimaB}[1]{{{#1}}}
\newcommand{\DimaC}[1]{{{#1}}}
\newcommand{\DimaD}[1]{{{#1}}}
\newcommand{\DimaE}[1]{{{#1}}}

\newcommand{\nextversion}[1]{} %?? Do not forget!

\newcommand{\proofend}{\hfill$\Box$\smallskip}

\begin{document}

\title{Generalized Whittaker quotients of Schwartz functions on $G$-spaces}
\author{Dmitry Gourevitch}
\address{Dmitry Gourevitch, Faculty of Mathematics and Computer Science, Weizmann
Institute of Science, POB 26, Rehovot 76100, Israel }
\email{dimagur@weizmann.ac.il}
\urladdr{http://www.wisdom.weizmann.ac.il/~dimagur}
\author{Eitan Sayag}
\address{Eitan Sayag,
 Department of Mathematics,
Ben Gurion University of the Negev,
P.O.B. 653,
Be'er Sheva 84105,
ISRAEL}
 \email{eitan.sayag@gmail.com}

%
%\classification{20G05,20G25,46F99}
%
%20-xx  Group theory and generalizations 20Cxx Representation
%theory of groups 20C99 None of the above, but in this section
%
%20Gxx Linear algebraic groups (classical groups) 20G05
%Representation theory 20G25 Linear algebraic groups over local
%fields and their integers
%
%
%22-xx  Topological groups, Lie groups 22Exx Lie groups 22E45
%Representations of Lie and linear algebraic groups over real
%fields: analytic methods {For the purely algebraic theory, see
%20G05}
%
%
%46-xx  Functional analysis 46Fxx Distributions, generalized
%functions, distribution spaces 46F10 Operations with distributions
%
%14-xx  Algebraic geometry 14Lxx Algebraic Groups 14L24 Geometric
%invariant theory [See also 13A50] 14L30 Group actions on varieties
%or schemes (quotients) [See also 13A50, 14L24]
%22E35          Analysis on $p$-adic Lie groups
\date{\today}

\begin{abstract}
Let $G$ be a reductive group over a local field $F$ of characteristic zero. Let $X$ be a $G$-space.
In this paper we study the existence of generalized Whittaker quotients for the space of Schwartz functions on $X$, considered as a representation of $G$. We show that the set of nilpotent elements of the dual space to the Lie algebra such that the corresponding generalized Whittaker quotient does not vanish contains the nilpotent part of the image of the moment map, and lies in the closure of this image. This generalizes recent results of Prasad and Sakellaridis.

Applying our theorems to symmetric pairs $(G,H)$ we show that there exists an infinite-dimensional $H$-distinguished  representation of $G$ if and only if the real reductive group corresponding to the pair $(G,H)$ is non-compact.
 For quasi-split $G$ we also extend to the Archimedean case the theorem of Prasad stating that  there exists a generic $H$-distinguished  representation of $G$ if and only if the real reductive group corresponding to the pair $(G,H)$ is quasi-split.

In the non-Archimedean case our result also gives rather sharp bounds on the wave-front sets of distinguished representations.

\DimaE{
The results in the present paper can be used to recover many of the vanishing results on periods of automorphic forms proved by Ash-Ginzburg-Rallis \cite{AGR}. This follows from our Corollary \ref{cor:global} when combined with the restrictions on the Whittaker support of cuspidal automorphic representations proven in  \cite{GGS2}.
}
\end{abstract}

\maketitle
% \begin{center}
% with an appendix
% by Alexander Kemarsky and Eitan Sayag\end{center}
% \begin{abstract}
% We state Prezbinda's conjecture.
% \begin{itemize}
%
% \item The wave-front set of the character-distribution $\Theta_{\pi}$ of an irreducible representation $\pi$ of $G$ over any point $g\in G$ is equal to the intersection of the wave-front set over the identity with the cenralizer of the point.
% %nilpotent cone of $T_g^*G \cong \fg$.
%
% \item Furthermore, let $H_1=H_2 \subset G$ be spherical subgroups. Consider a $\fz$-finite distribution $\xi$ of $G$ that is $H_1\times H_2$-invariant distribution on $G$. Then $$WF_{g}(\xi)=WF_{1}(\xi) \cap stab_{H}(g)$$
% \item Future: Let $H$ act on two spherical space $X=X_1=X_2.$
% %We prove that the $\fz$-finite distributions in $\cJ$ form a dense subspace. In fact we prove this result in wider generality, where the groups $H_i$ are spherical groups of certain type and the invariance condition is replaced by equivariance.
% \end{itemize}
%\end{abstract}

%\tableofcontents

\section{Introduction}\label{sec:intro}

Let $F$ be a local field of characteristic zero.
Let $\bf G$ be a \DimaD{connected algebraic} reductive group defined over $F$, let $G:={\bf G}(F)$ be its $F$-points and $\fg$ be the Lie algebra of $G$.

If $F$ is non-Archimedean, we denote by $\cM(G)$ the category of admissible smooth finitely-generated representations (\cite{BZ}). If $F$ is Archimedean, we denote by $\cM(G)$ the category of Casselman-Wallach representations, {\it i.e.}  admissible smooth finitely-generated \Fre representations of moderate growth (see \cite[\S 11]{Wal},  or \cite{CasGlob}). We denote by $\Irr(G)$ the collection of irreducible representations in $\cM(G)$.

A classical theme in representation theory of reductive groups over local fields is the study of representations {\it distinguished} with respect to a subgroup $H \subset G.$ A representation $\pi\in \Irr(G)$ is called $H$-distinguished if it has a non-zero $H$-equivariant  linear functional  (continuous if $F$ is Archimedean). Denote by $\Irr(G)_{H}\subset \Irr(G)$ the subcollection of $H$-distinguished representations.
Such are the local components of {\it distinguished} automorphic representations and such are the discrete series representations that  contribute to the Plancherel decomposition of the Hilbert space $L^{2}(X)$, the space of square integrable functions (or sections of density bundle) on the $G$-space $X=G/H.$
%when this space admits an invariant measure.

In this paper we offer a unified approach to two seemingly unrelated questions concerning distinguished representations. The first is to clarify and generalize the relationship between distinction and genericity studied in \cite{PraSak} to arbitrary $G$-spaces over local fields, Archimedean or not. The second is a search for a non-Archimedean analogue of the qualitative study of the Plancherel decomposition provided by \cite{HarWeich}.

Regarding the first question, it was recently studied in the special case of linear periods in \cite{SayVer} both for Archimedean and non-Archimedean fields, and for symmetric pairs over non-Archimedean fields in \cite{PraSak}. It should be mentioned that the motivating example here is the disjointness result of \cite{HeuRal} showing that irreducible generic representations of $GL(2n)$ never have non-zero symplectic invariant functionals (see \cite{OS} for a complete account).

As for the second, while a description of the wavefront set of the unitary representation $L^{2}(X)$ is of independent interest, it also yields a-priori bound on the wavefront of individual irreducible unitary distinguished representation that occurs in the Plancherel decomposition.

The unification of the topics mentioned above is obtained by studying the {\it Whittaker support} (\cite{GGS2}) of the non-admissible representation of $G$ on spaces of Schwartz functions on $X.$ Understanding Whittaker support allows, in the non-Archimedean case, to deduce exact information on wave front of individual distinguished representations leading to our answer to the first question. The study of $ \WO(\Sc(X))$, which is our smooth replacement to the problem studied in \cite{HarWeich}, is carried out using the theory of invariant distributions and in particular the {\it orbitwise} technique introduced by Gelfand-Kazhdan \cite{GK} in the non-Archimedean case (and its various extensions and ramifications), that in some cases reduces the study of invariant distributions on a space to the study of invariant distributions on each of the orbits separately.

 In this language, the orbits that can support equivariant distributions are those for which a certain character is trivial on a stabilizer of one (hence any) point in the orbit. This can be reformulated as a condition tangling the point and the character as living in the image of a partial moment map attached to the $G$-space $X$ and the orbit.
Thus, as will also be clear from the formulation of our results, a key tool in our approach is the moment map attached to the $G$-space $X$.
%Mention also Yiannis.
%Here we use deep results of Knop \cite{Kno}.
An unexpected aspect of this geometric approach is that it allows us to study the Whittaker supports of modules of the form $\Sc(X)$ even in cases where $X$ is not $G$-homogenous.
% Aspect

% by giving an apriori description of the wavefront set of the unitary %representation $L^{2}(X).$

%\subsection{Main results}

%Let $F$ be a local field of characteristic zero, Archimedean or not.
%Let $\bf G$ be a reductive group defined over $F$, let $G:={\bf G}(F)$ be %its $F$-points and $\fg$ be the Lie algebra of $G$.
We now turn to the results of the present paper.  We need some notations.
Let $\bf X$ be a smooth  $\bf G$-variety defined over $F$, and let $X:={\bf X}(F)$.
%?? We do not absulutely need the smoothness assumption. If $F$ is Archimedean we assume in addition that $X$ is smooth.
Let $\mu:T^*X\to \fg^*$ denote the moment map, and let $\cM$ denote the closure of its image. Note that $\cM$ is a closed conical set. Let $\cN\subset \fg^*$ denote the nilpotent cone.
Let $\Sc(X)$ denote the space of Schwartz functions on $X$ and let $\WO(\Sc(X))\subset \cN$ denote the set of nilpotent elements $\varphi$ such that $\Sc(X)$ has a generalized Whittaker quotient corresponding to $\varphi$ (see \S \ref{subsec:Whit} below).
In \S \ref{sec:PfWO} we prove the following theorem.

\begin{introthm}[\S \ref{sec:PfWO}]\label{thm:WO}\begin{enumerate}[(i)]
\item We have
$ \WO(\Sc(X))\subset \cM\cap \cN$.
\item If either $F$ is non-Archimedean or $\bf X$ is quasi-projective then $\Im \mu \cap \cN \subset \WO(\Sc(X))$.
\end{enumerate}
\end{introthm}

Let us now consider the homogeneous case, {\it i.e.} $X=G/H$. In this case we have $\Im \mu=G\cdot\fh^{\bot}$. Here, $\fh^{\bot}$ denotes the orthogonal complement to $\fh$ in $\fg^*$, and $G\cdot\fh^{\bot}$ denotes the image of $\fh^{\bot}$ under the coadjoint action of $G$ on $\fg^*$. For homogeneous $X$ we can formulate a twisted version of Theorem \ref{thm:WO}.

Let $\zeta:H\to F$ and $\psi:H\to F^{\times}$ be algebraic characters.
Let $d\zeta\in \fh^*$ denote the differential of $\zeta$, and let $Fp_{\fh}^{-1}(d\zeta)
\subset \fg^*$ denote the linear space spanned by the preimage of $d\zeta$ under the restriction map
$p_{\fh}:\fg^*\onto \fh^*$.
Fix a character $\xi_m$ of $F^{\times}$ and a non-trivial unitary character $\xi_a$ of $F$, and let $\chi$ be the character of $H$ given by the product
\begin{equation}\label{=chi}
\chi=(\xi_a\circ \zeta)\cdot(\xi_m\circ\psi).
\end{equation}
Let $\Sc(G)_{H,\chi}$ denote the space of $\chi$-coinvariants in $\Sc(G)$ under the action of $H$ by right multiplications (see \eqref{=coinv} below).
\begin{introthm}[\S \ref{sec:PfWO}]\label{thm:Twisted}
Let $H$,  $\chi$, and $Fp_{\fh}^{-1}(d\zeta)$ be as above. Then
$$G\cdot p_{\fh}^{-1}(d\zeta)\cap \cN \subset \WO(\Sc(G)_{H,\chi})\subset \overline{G\cdot Fp_{\fh}^{-1}(d\zeta)}\cap \cN.$$
\end{introthm}

%Our interest in $\Sc(X)$, and in $\Sc(G)_{H,\chi}$ comes from the study of %admissible representations. If $F$ is non-Archimedean, we denote by %$\cM(G)$ the category of admissible smooth finitely-generated %representations. If $F$ is Archimedean, we denote by $\cM(G)$ the %category of Casselman-Wallach representations, {\it i.e.}  admissible %smooth finitely-generated \Fre representations of moderate growth (see %\cite[\S 11]{Wal},  or \cite{CasGlob}). We denote by $\Irr(G)$ the %collection of irreducible representations in $\cM(G)$. A representation %$\pi\in \Irr(G)$ is called $(H,\chi)$-distinguished if it has a non-zero $(H,%\chi)$-equivariant  linear functional  (continuous if $F$ is Archimedean). %Denote by $\Irr(G)_{H,\chi}\subset \Irr(G)$ the subcollection of $(H,%\chi)$-distinguished representations.

 By Frobenius reciprocity for small induction  (see {\it e.g.} \cite[Proposition 2.29]{BZ} and \cite[Lemma 2.3.4]{GGS}), $\pi\in \Irr(G)$ is $(H,\chi)$-distinguished if and only if the contragredient representation $\widetilde{\pi}$ is a quotient of $\Sc(G)_{H,\chi\Delta_H}$, where $\Delta_H$ denotes the modular function.
\DimaD{This allows us to deduce from Theorem \ref{thm:Twisted} the following corollary.
%\begin{equation}\label{=Irr}
\begin{introcor}[\S \ref{subsec:PfIrr}]\label{cor:Irr}
Let $H$ and  $\chi$ be as above, and let $\pi\in \Irr(G)_{H,\chi}$. Then
$$\WO(\pi)\subset \overline{G\cdot Fp_{\fh}^{-1}(d\zeta)}.$$
\end{introcor}
%\end{equation}
For the case $G=\GL_{n+k}(F)$ and $H=\GL_n(F)\times \GL_k(F)$, and trivial $\chi$, this statement is equivalent to  \cite[Theorem B]{SayVer}, which is proven for all local or finite fields $F$ of characteristic different from 2.
Corollary \ref{cor:Irr} is especially useful for non-Archimedean $F$, since for them the top nilpotent orbits in $\WO(\pi)$ coincide with the top nilpotent orbits in the wave-front set $\overline{\WF}(\pi)$ by \cite{MW}. In \S \ref{sec:WF} below we recall this notion and deduce from Corollary \ref{cor:Irr} the following corollary.

\begin{introcor}[\S \ref{sec:WF}]\label{cor:WF}
Let $H$ and $\chi$ be as above, and let $\pi\in \Irr(G)_{H,\chi}$. Suppose that $F$ is non-Archimedean.  Then $\overline{\WF}(\pi)\subset\overline{G\cdot Fp_{\fh}^{-1}(d\zeta)}.$

In particular, if $\chi$ is trivial on the unipotent radical of $H$ then $\overline{\WF}(\pi)\subset \overline{G\cdot\fh^{\bot}}$.
\end{introcor}
}

This corollary was our starting point for this paper. It was inspired by the paper \cite{HarWeich} that shows that the wave-front set of $L^2(G/H)$ in the Archimedean case is $\overline{G\cdot\fh^{\bot}}$.

Another powerful Archimedean analogue of Corollary \ref{cor:WF} is proven in \cite{GS_R}. Namely, \DimaB{in \cite{GS_R} we show that if $H$ is a spherical subgroup of $G$\ then} the Zariski closure of $\overline{\WF}$ coincides with the Zariski closure of a nilpotent $\bfG$-orbit in $\fg^*(\C)$ that intersects $\fh^\bot$. We conjecture that the same holds in the $p$-adic case. We refer the reader to \cite[\S 10]{GS_R} for more details on this conjecture and its generalization, as well as some evidence and applications.

Next we ask the following question: given  $\varphi\in \WO(\Sc(G)_{H,\chi})$, does there exist $\pi\in \Irr(G)_{H,\chi}$ with $\varphi\in \WO(\pi)$?
We can answer this question for certain $\varphi$, and for (absolutely) spherical subgroups $H$, {\it i.e.} subgroups $H={\bf H}(F)$ with $\bf H$  acting on the flag variety of $\bf G$ with an open orbit. Let $P$ be an adapted parabolic for $G/H$ - see \S \ref{sec:Geo} below for this notion.
For any character $\chi$ of $H$, denote by $\chi'_P$ the character of $P\cap H$ given by
\begin{equation}
\chi'_{P}:=\chi|_{P\cap H}\Delta_{P}^{-1/2}\Delta_H^{-1}\,,
\end{equation}
where $\Delta$ denotes the modular functions.
\begin{introprop}[\S \ref{sec:PfCor}]\label{prop:PrinSer}
Let $P$ be an adapted parabolic of $G/H$ such that $PH$ is open  in $G$ and let $\pi=\Sc(G)_{P,\Delta_P^{-1/2}}$ be the space of smooth vectors of the normalized induction $\mathrm{Ind}_{P}^G1$. Then
\begin{enumerate}[(i)]
\item \label{PS:irr} $\pi$ is irreducible and unitarizable.
\item \label{PS:WO} $\WO(\pi)=G\cdot \fp^{\bot}$
\item \label{PS:Quot} Assume $H$ is spherical, and $\chi$ is a character of $H$. If $\chi$ is trivial on the unipotent radical of $H$ and $\chi'_{P}=1$   then $\pi\in \Irr(G)_{H,\chi}$.
\end{enumerate}
\end{introprop}
We note that if $H$ is unimodular then $\Delta_P=\Delta_H=1$, and $\chi'_P=\chi|_P$ (see Lemma \ref{lem:uni} below).

%??Want to update here when we have a better geometry statement.

\begin{remark}
In case $X$ is a unimodular spherical homogeneous space and if in addition
$\cM\cap \cN\subset G\cdot\fp^{\bot}$ then
Theorem \ref{thm:WO} and Proposition \ref{prop:PrinSer}
imply the exact equality
\begin{equation}\label{=WO}
%\text{ for any }\pi\in \Irr(G)_{H,\chi},\, \WO(\pi)\subset %\overline{G\cdot\fh^{\bot}}\,\,.
\WO(\Sc(X))=\cM\cap \cN=G\cdot\fp^{\bot}
\end{equation}

%Which we consider as the smooth analogue of \cite{Har}
%It is possible that $\cM\cap \cN\subset G\cdot\fp^{\bot}$.
% If this holds then Theorem \ref{thm:WO} and Proposition %\ref{prop:PrinSer} imply that for unimodular spherical homogeneous $X$ %we have $\cM\cap \cN=G\cdot\fp^{\bot}=\WO(\Sc(X))$.
We note that if $G$ is quasi-split then  \cite{Kno} and \cite[Appendix A]{PraSak} imply the weaker inclusion $\cM\cap \cN\subset \bfG\cdot\fp^{\bot}$ - see Corollary \ref{cor:Geo} below.
\end{remark}

\begin{introcor}[\S \ref{sec:PfCor}]\label{cor:main} Let $H\subset G$ be an algebraic subgroup, and let $\chi$ be a character of $H$, trivial on its unipotent radical. Let $P$ be an adapted parabolic of $G/H$.
\begin{enumerate}[(i)]
\item \label{cormain:FinDim}If $P=G$ then every   $\pi\in \Irr(G)_{(H,\chi)}$ is finite-dimensional.
\item \label{cormain:ExFinDim} If $P\neq G$,  $PH$ is open  in $G$, $H$ is (absolutely) spherical and $\chi'_{P}=1$\\ then there exists an infinite-dimensional unitarizable  $\pi \in \Irr(G)_{(H,\chi)}$.
\item \label{cormain:ExGen}If $P$ is a Borel subgroup of $G$ (in particular $G$ is quasi-split), and $\chi'_{P}=1$\\ then  there exists a generic unitarizable  $\pi \in \Irr(G)_{(H,\chi)}$.
\item \label{GH:NoGen} If $P$ is not a Borel subgroup of $G$ then no  $\pi\in \Irr(G)_{(H,\chi)}$ is generic.
\end{enumerate}
\end{introcor}

% For quasi-split $G$ and spherical $H$, the adapted parabolic $P$ is a Borel subgroup if and only if the dimension of the flag variety equals $\dim H-\dim (H\cap Z),$ where $Z$ is   the center of $G$. This follows from the local structure theorem (Theorem \ref{thm:LS} below).  In this case $H$ is called a minimal spherical subgroup.

\DimaB{Over Archimedean $F$, Parts \eqref{cormain:FinDim} and \eqref{GH:NoGen}
follow from \cite[Theorem 4.2]{KS}.
Over non-Archimedean $F$, }Part \eqref{GH:NoGen} was essentially shown in \cite{PraSak}. The emphasis in  \cite{PraSak} is on symmetric pairs, for which Prasad showed that some properties of distinguished representations are governed by the properties of the real reductive group given by the root system of the symmetric pair.
In support of this principle we derive the following corollary.

\begin{introcor}[\S \ref{sec:PfCor}]\label{cor:sym}
Suppose that $G$ is quasi-split and let $H\subset G$ be a symmetric subgroup.\\
Let $G^H_{\R}$ be the real reductive group corresponding to the symmetric pair $(G,H)$. Then
\begin{enumerate}[(i)]
\item \label{Sym:Gen} There exists a generic $\pi\in \Irr(G)_H$ if and only if $G^H_{\R}$ is quasi-split.
\item There exists an infinite-dimensional $\pi\in \Irr(G)_H$ if and only if $G^H_{\R}$ is not a quotient of the product of a compact group and a torus.
\end{enumerate}
\end{introcor}
Part \eqref{Sym:Gen} was proven in \cite{PraSak} for non-Archimedean local fields of arbitrary characteristic, as well as a certain analogue for finite fields.

In \cite[\S 6]{GS_R} we apply this corollary in order to answer \cite[Question 1]{PraSak} in many cases.

\DimaC{Finally, let us give an application to automorphic periods.
Suppose that $\bf G$ and $\bf H$ are defined over a number field $\mathbb{K}$, and let ${\bf G}(\mathbb{A}_\mathbb{K})$ and ${\bf H}(\mathbb{A}_\mathbb{K})$ denote their adelic points.
For any automorphic representation $\pi$ we define $\WO(\pi)\subset \fg^*(\mathbb{K})$ to be the set of all nilpotent $\varphi\in \fg^*(\mathbb{K})$ such that $\pi$ possesses a non-zero continuous $({\bf N}(\mathbb{A_{\mathbb{K}}}),\eta_{\varphi})$-equivariant functional, where the nilpotent subgroup ${\bf N}\subset \mathbf{G}$ and the character $\eta_{\varphi}$ of ${\bf N}(\mathbb{A_{\mathbb{K}}})$  are defined as is Definition \ref{def:Whit} below.
\begin{introcor}[\S\ref{sec:global}]\label{cor:global}
 Let $\pi$ be an irreducible  smooth automorphic representation of ${\bf G}(\mathbb{A}_\mathbb{K})$ that possesses a non-zero continuous ${\bf H}(\mathbb{A}_\mathbb{K})$-invariant functional. Then for every place $\nu$ of $\mathbb{K}$ we have
\begin{equation}\label{=aut}
\WO(\pi)\subset \overline{{\bf G}(\mathbb{K}_{\nu})\cdot\fh^{\bot}(\mathbb{K}_{\nu})},
\end{equation} where the closure is taken in the local topology of $\fg^*(\mathbb{K}_{\nu})$.
\end{introcor}

\DimaD{If one allows the functional to be defined only on $K$-finite vectors in $\pi$, the inclusion \eqref{=aut} will still hold for all non-archimedean places $\nu$. }

The main example of an ${\bf H}(\mathbb{A}_\mathbb{K})$-invariant functional is a period integral over ${\bf H}(\mathbb{K})\backslash {\bf H}(\mathbb{A}_\mathbb{K})$, or its invariant regularization. Such an integral is known to converge absolutely when $\bfH$ is a symmetric subgroup of $\bfG$ and $\pi$ is cuspidal \cite[Proposition 1]{AGR}. Together with the restriction on the Whittaker support of cuspidal representations given in \cite[Theorem 8.4(ii) and Corollary 8.11]{GGS2}, this gives an alternative proof for many of the results of \cite{AGR}.}

%\subsection{Structure of the paper}??

\subsection{Acknowledgements}

We thank Avraham Aizenbud, \DimaB{Itay Glazer,} and Michal Zydor for fruitful discussions.
D. G. was partially supported by ERC StG 637912 and ISF  249/17.

%
% \begin{cor}\label{cor:main}
% If $X$ is unimodular then $\cM\cap \cN=G\fp^{\bot}=\WO(\Sc(X))$.
% \end{cor}
% ?? Can't believe this  depends on unimodularity. Probably should indeed use half-densities.

\section{Preliminaries}

\subsection{Smooth representations and generalized Whittaker quotients}\label{subsec:Whit}
% For non-archimedean $F$ we will work with  $l$-groups, {\it i.e.} Hausdorff topological groups having a basis for the topology at the identity consisting of open compact subgroups. This generality includes $F$-points of algebraic groups defined over $F$. For Archimedean $F$ we will work with affine Nash groups, {\it i.e.} Lie groups that are given in $\R^n$ by semi-algebraic equations, as well as the graphs of their multiplication maps.  This generality includes $F$-points of algebraic groups defined over $F$, and their finite covers (see \cite{dCl,AG,Sun,FS}).

\begin{notn}\label{not:rep}
%If $G$ is an $l$-group,
If $F$ is non-Archimedean, we denote by $\Rep^{\infty}(G)$ the category of smooth representations of $G$ in complex vector spaces (see {\it e.g.} \cite{BZ}).
%For $V,W\in \Rep^{\infty}(G)$, $V\otimes W$ will denote the usual tensor product over $\C$. We denote by $\cM(G)\subset \Rep^{\infty}(G)$ the subcategory consisting of representations of finite length.

If $F$ is Archimedean,
%an affine Nash group,
we denote by $\Rep^{\infty}(G)$ the category of smooth  \Fre representations of $G$ of moderate growth, as in \cite[\S 1.4]{dCl}.
%?? I canceled the nuclear requirement
%?? with the additional assumption that the representation spaces are nuclear (see e.g. \cite[\S 50]{Tre}). For $V,W\in \Rep^{\infty}(G)$, $V\otimes W$ will denote the completed projective tensor product. We denote by $\cM(G)\subset \Rep^{\infty}(G)$ the subcategory consisting of admissible finitely generated representations, see \cite[Ch. 11]{Wal}. This category is abelian, see \cite[Ch. 11]{Wal} or \cite{CasGlob}.
\end{notn}

For any algebraic subgroup $H\subset G$ and $\pi\in \Rep^{\infty}(G)$, denote by $\pi_H$ the space of coinvariants, i.e. quotient of $\pi$ by the intersection of kernels of all $H$-invariant functionals. Explicitly,
\begin{equation}\label{=coinv}
\pi_H=\pi/\overline{\{\pi(g)v -v\, \vert \,v\in \pi, \, g\in H\}}
\end{equation}
where the closure is needed only for Archimedean $F$. In the latter case,
for connected $H$ we have  $\pi_H=\pi/\overline{\fh_{\C}\pi}$ which in turn is equal to the quotient of $\oH_0(\fh,\pi)$ by the closure of zero, where $\fh$ denotes the Lie algebra of $H$. For any character $\chi$ of $H$, denote $\pi_{H,\chi}:=(\pi\otimes \chi)_{H}$.
%?? Inverse?

%?? define $\fg$ and the moment map.

\begin{defn}\label{def:Whit}
Let  $\varphi\in \fg^*$ be a nilpotent element, and let $\pi\in \Rep^{\infty}(G)$. We define the generalized Whittaker quotient $\pi_{\varphi}$ in the following way.

Choose an $\sll_2$-triple $(e,h,f)$ such that $\varphi$ is given by the Killing form pairing with $f$. Now, let $\fg^h_1$ denote the eigenspace of the adjoint action of $h$ on $\fg$ corresponding to eigenvalue 1, and $\fg^h_{\geq 2}$ denote the sum of the eigenspaces with eigenvalues 2 and higher. Consider the symplectic form $\omega_{\varphi}$ on $\fg^h_1$ given by $\omega_{\varphi}(X,Y):=\varphi([X,Y])$ and choose a Lagrangian $\fl$ for this form. Let $\fn$ be the nilpotent Lie algebra $\fl\oplus \fg^h_{\geq 2}$ and $N\subset G$ be the corresponding unipotent subgroup. Let $\eta_{\varphi}$ denote the unitary character of $N$ given by $\varphi$. Then we define $\pi_{\varphi}:=\pi_{N,\eta_{\varphi}}$.
\end{defn}
For more details about this definition, and the proof that it is independent of choices, we refer the reader to \cite[\S 2.5]{GGS2}.

\begin{defn}
Let  $\pi\in \Rep^{\infty}(G)$. Define $\WO(\pi)$ to be the set of all nilpotent elements $\varphi \in \fg^*$ satisfying $\pi_{\varphi}\neq 0$.
\end{defn}

\subsection{Invariant distributions}
Let $X$ be the manifold of $F$-points of a smooth  algebraic variety defined over $F$. Let $\Sc(X)$ denote the space of Schwartz functions on $X$. For non-Archimedean $F$, this means the space of locally-constant compactly supported functions, see \cite{BZ}. For Archimedean $X$, Schwartz functions are functions that decay rapidly together with all their derivatives, see {\it e.g.} \cite{AG} for the precise definition. Let $\Sc^*(X)$ denote the linear dual space if $F$ is non-Archimedean, and continuous linear dual space if $F$ is Archimedean. We refer to elements of $\Sc^*(X)$ as tempered distributions. For a group $H$ acting on $X$ and its character $\chi$, we denote by $\Sc^*(X)^{H,\chi}$ the space of tempered distributions $\xi$ satisfying $h\xi=\chi(h)\xi$. This space is dual to $\Sc(G)_{H,\chi}$.

%?? Change formulation and proof here.

% We will need the following theorem, that follows from \cite[Theorem A]{GSS}, \cite[Corollary 2.2.16]{AG_ST}, \cite[\S 1.5]{Ber} and \cite[\S 6]{BZ}.
\begin{thm}\label{thm:BZ}
Let an $F$-algebraic group $H$ act algebraically on $X$.
Let $\chi$ be a character of $H$ of the form \eqref{=chi}.
Suppose that the stabilizer $H_x$ of any point $x\in X$ is unipotent.
\begin{enumerate}[(i)]
\item \label{BZ:Non}If $\Sc^*(X)^{H,\chi}\neq 0$ then there exists $x\in X$ such that $\chi|_{H_x}=1$.
\item \label{BZ:Ex}If there exists $x\in X$ such that $\chi|_{H_x}=1$ and one of the following holds
\begin{enumerate}[(a)]
\item $F$ is non-Archimedean
\item \label{BZ:Arch}$H$ is solvable and $X$ is quasi-projective
\end{enumerate}
Then $\Sc^*(X)^{H,\chi}\neq 0$.
\end{enumerate}
\end{thm}
\begin{proof}
Since all algebraic characters of unipotent groups are trivial, we have $\Delta_{H}|_{H_x}=\Delta_{H_x}=1$
for any $x\in X$. Now, \eqref{BZ:Non} follows from \cite[\S 1.5]{Ber} and \cite[\S 6]{BZ} for non-Archimedean $F$, and from \cite[Theorem 2.2.15 and the proof of Corollary 2.2.16]{AG_ST} for Archimedean $F$.

For \eqref{BZ:Ex}, note that the condition $\chi|_{H_x}=1$ implies that the orbit $Hx$ of $x$ has an $H$-invariant measure. The Archimedean case under the condition \eqref{BZ:Arch} follows now from   \cite[Theorem A]{GSS}. Suppose now that $F$ is non-Archimedean, and let $\cO$ be an orbit of minimal dimension among those that have an $(H,\chi)$-equivariant measure. Since all the stabilizers are unipotent, and thus Zariski connected, \cite[Theorem 1.4]{HS} implies  that  $\Sc^*(\overline{\cO})^{H,\chi}\neq 0$. This in turn implies $\Sc^*(X)^{H,\chi}\neq 0$.
\end{proof}

To complement this theorem in the Archimedean case we will need the following one.

\begin{theorem}[{cf. \cite[Theorem B]{GSS}}]\label{thm:UGH}
Let $G$ be a real reductive group, and let  $H,N\subset G$ be  real algebraic subgroups, such that $N$ is unipotent. Let $\chi$ be a character of $H$ as in \eqref{=chi}, and let $\eta$ be a unitary character of $N$. Assume that for some $g\in G$ we have $$\chi|_{H\cap N^g}=\eta^g|_{H\cap N^g},$$
where $N^g=gNg^{-1}$ and $\eta^g$ is the character of $N^g$ given by $\eta^g(gxg^{-1})=\eta(x)$. Then
$$\Sc^*(G)^{N\times H,\eta \times \chi}\neq 0.$$
\end{theorem}

\section{Proof of Theorems \ref{thm:WO} and \ref{thm:Twisted}, and Corollaries \ref{cor:Irr} and \ref{cor:WF}}\label{sec:PfWO}

Let $\varphi\in \fg^*$ be a nilpotent element. As in Definition \ref{def:Whit}, let $(e,h,f)$ be an $\sll_2$-triple in $\fg$ such that $\varphi$ is given by the Killing form pairing with $f$.
Let $\fv:=\fg^h_{\geq 2}$, and let $\fn\supset \fv$ be as in Definition \ref{def:Whit}. Let $V:=\Exp(\fv)\subset N:=\Exp(\fn)$ be the corresponding nilpotent subgroups of $G$. Let $\eta_{\varphi}$ be the unitary character of $N$ corresponding to $\varphi$.

\begin{proof}[Proof of Theorem \ref{thm:WO}]
Suppose that $\varphi\in \WO(\Sc(X))$.
Then by definition of $\WO(\Sc(X))$, we have $\Sc^*(X)^{N,\eta_{\varphi}}\neq 0$ and thus $\Sc^*(X)^{V,\eta_{\varphi}}\neq 0$ . By Theorem \ref{thm:BZ}, this implies that there exists $x\in X$ such that $\varphi$ vanishes on the Lie algebra $\fv_x$ of the stabilizer of $x$ in $V$.
Thus, the restriction $\varphi|_\fv$ lies in the image of the moment map $\mu_V:T_x^*X\to \fv^*$ which is the dual of the map $\fv\to T_xX$ given by the differential of the action of $V$ on $X$. But $\mu_V$ equals the composition of $\mu:T_x^*X\to \fg^*$ with the restriction $\fg^*\to \fv^*$:
\begin{equation}
T_x^*X\overset{\mu}{\to}\fg^*\onto \fv^*
\end{equation}
Thus there exists $\psi\in \cM$ such that $\psi|_\fv =\varphi$. Since the kernel of the restriction to $\fv$ is $(\fg^*)^h_{\geq -1}$, we have  $\psi=\varphi+\psi'$, where $\psi'\in (\fg^*)^h_{\geq -1}$. Let $T\subset G$ be the one-dimensional torus with Lie algebra spanned by $h$, and let $t\in T$ with $|t|<1$. Then the coadjoint action $Ad^*(t)$ of $t$ is given on $(\fg^*)^h_i$ by $t^i$. Define a sequence $\psi_n\in \fg^*$ by $\psi_n:=t^{2n}\Ad^*(t^n)\psi$. Then $\psi_n$ converges to $\varphi$. Since $\cM$ is conic and $G$-invariant, $\psi_n\in \cM$. Since $\cM$ is closed, we have $\varphi \in \cM$.

Conversely, if $\varphi$ lies in the image of the moment map  then $\varphi$ vanishes on $\fg_x$ for some $x\in X$ and thus $\eta_{\varphi}$ is trivial on $N_x$. By Theorem \ref{thm:BZ} this implies $\Sc^*(X)^{N,\eta_{\varphi}}\neq 0$, which by definition means $\varphi \in \WO(\Sc(X))$.
\end{proof}
Let us now prove Theorem \ref{thm:Twisted} in a similar way.
\begin{proof}[Proof of Theorem \ref{thm:Twisted}]

Suppose that $\varphi\in \WO(\Sc(G)_{H,\chi})$.
Then by definition we have $\Sc^*(G)^{N\times H,\eta_{\varphi}\times \chi}\neq 0$ and thus $\Sc^*(G)^{V\times H,\eta_{\varphi}\times \chi}\neq 0$. By Theorem \ref{thm:BZ}, this implies that there exists $g\in G$ such that $\eta_{\varphi}\times \chi$ vanishes on the stabilizer in $V\times H$ of $g$. Replacing $\varphi$ by $g^{-1}\cdot \varphi$ we can assume that $g$ is the unit element. Then the stabilizer can be identified with $V\cap H$, and  we have
 $\varphi|_{\fv\cap \fh}=d\zeta|_{\fv\cap \fh}$. Thus there exists $\psi\in \fg^*$ such that $\psi|_{\fv}=\varphi|_{\fv}$ and $\psi|_{\fh}=d\zeta$.

Since the kernel of the restriction to $\fv$ is $(\fg^*)^h_{\geq -1}$, we have  $\psi=\varphi+\psi'$, where $\psi'\in (\fg^*)^h_{\geq -1}$. Let $T\subset G$ be the one-dimensional torus with Lie algebra spanned by $h$, and let $t\in T$ with $|t|<1$.  Define a sequence $\psi_n:=t^{2n}\Ad^*(t^n)\psi\in \fg^*$. Then $\psi_n\to\varphi$ and thus $\varphi\in\overline{G\cdot Fp_{\fh}^{-1}(d\zeta)}\cap \cN.$

Conversely, if $\varphi\in G\cdot p_{\fh}^{-1}(d\zeta)\cap \cN$ then $\chi|_{H\cap N^g}=\eta_{\varphi}^g|_{H\cap N^g}$ for some $g\in G$.
By Theorem \ref{thm:UGH} for Archimedean $F$, and Theorem \ref{thm:BZ} for non-Archimedean $F$, this implies $\Sc^*(G)^{N\times H,\eta_{\varphi}\times \chi}\neq 0$ and thus $\varphi\in \WO(\Sc(G)_{H,\chi})$.
\end{proof}
\DimaD{
\subsection{Proof of Corollary \ref{cor:Irr}}\label{subsec:PfIrr}
For the proof we will need the following proposition \DimaE{whose proof is postponed to \S\ref{subsec:-WO}.}
\begin{prop}\label{prop:-WO}
For any $\pi\in \Irr(G)$ we have $\overline{\WO(\pi)}=-\overline{\WO(\widetilde{\pi})}$.
Moreover, if $F$ is Archimedean then $\WO(\pi)=-\WO(\widetilde{\pi})$.
\end{prop}
\begin{proof}[Proof of Corollary \ref{cor:Irr}]
Since the Whittaker quotient is a space of coinvariants, the functor $\pi\mapsto \pi_{\varphi}$ is right-exact. If $\pi$ is $H$-distinguished then $\widetilde{\pi}$ is a quotient of $\Sc(G)_{H,\Delta_H\chi}$ and thus
$\WO(\widetilde{\pi})\subset \WO(\Sc(G)_{H,\Delta_H\chi})$.
By theorem \ref{thm:Twisted} we have  $\WO(\Sc(G)_{H,\Delta_H\chi}) \subset\overline{G\cdot Fp_{\fh}^{-1}(d\zeta)}$. By Proposition \ref{prop:-WO} we have $\overline{\WO(\widetilde{\pi})}=-\overline{\WO(\pi)}$.
Altogether we obtain
$$\WO(\pi)\subset-\overline{\WO(\widetilde{\pi})}\subset- \overline{\WO(\Sc(G)_{H,\Delta_H\chi})}\subset \overline{G\cdot Fp_{\fh}^{-1}(d\zeta)}$$
\end{proof}

\subsection{Wave-front sets and the proof of Corollary \ref{cor:WF}} \label{sec:WF}
In this subsection only we assume that $F$ is non-Archimedean. Let $\pi\in \Irr(G)$. Let $\chi_{\pi}$ be the character of $\pi$. It is a generalized function on $G$ and it defines a generalized function $\xi_{\pi}$ on a neighborhood of zero in ${\mathfrak{g}}$,
by restriction to a neighborhood of $1\in G$ and applying logarithm.
By \cite{HowGL} and \cite[p. 180]{HCWF},  $\xi_{\pi}$ is  a linear combination of Fourier transforms of $G$-invariant measures of nilpotent coadjoint orbits.
The measures  extend to $\fg^*$ by \cite{RangaRao}.
Denote by $\overline{WF}(\pi)$ the closure of the union of all the orbits that appear in the decomposition of $\xi_\pi$ with non-zero coefficients.

\begin{thm}[{\cite[Proposition I.11, Theorem I.16 and Corollary I.17]{MW}, and \cite{Var}}] \label{thm:MW}$\,$\\
For any $\pi\in \Irr(G)$ we have
$\overline{\WF}(\pi)=\overline{\WO(\pi)}$.
\end{thm}
%Corollary \ref{cor:WF} immediately follows from this theorem and
\begin{proof}[Proof of Corollary \ref{cor:WF}]
By Corollary \ref{cor:Irr} we have $\WO(\pi)\subset \overline{G\cdot Fp_{\fh}^{-1}(d\zeta)}$. By Theorem \ref{thm:MW} we have $\overline{\WF}(\pi)=\overline{\WO(\pi)}$.
Altogether we obtain
$\overline{\WF}(\pi)=\overline{\WO(\pi)}\subset \overline{G\cdot Fp_{\fh}^{-1}(d\zeta)}.$
If $\chi$ is trivial on the unipotent radical of $H$ then $\zeta=1$ and
$\overline{G\cdot Fp_{\fh}^{-1}(d\zeta)}=\overline{G\cdot \fh^{\bot}}.$
\end{proof}

\subsection{Proof of Proposition \ref{prop:-WO}}\label{subsec:-WO}
For non-Archimedean this follows from Theorem \ref{thm:MW}. Indeed,
in this case by Theorem \ref{thm:MW} we have $\overline{\WF}(\widetilde{\pi})=\overline{\WO(\widetilde{\pi})}$, and it is easy to see that $\overline{\WF}(\widetilde{\pi})=-\overline{\WF}(\pi)$. Thus $\overline{\WO(\widetilde{\pi})}=-\overline{\WO(\pi)}$.

If $F$ is Archimedean then by \cite{HS} (cf. \cite{Ada}) there exists a canonical involution $\theta$ of $G$, called the Chevalley involution, such that $\theta(g)$ is conjugate to $g^{-1}$ for any $g\in G$, and $\alp \circ d\theta\in -G\cdot \alp$ for any $\alp\in \fg^*$. By \cite[Corollary 1.4]{Ada}  the twisted representation $\pi^{\theta}$ is isomorphic to $\widetilde{\pi}$.
Thus
$\WO(\widetilde{\pi})=\WO(\pi^{\theta})=d\theta(\WO(\pi))=-\WO(\pi).$
\proofend
} % end od \DimaD
\section{Preliminaries on the geometry of $X$ }\label{sec:Geo}
In this section we will identify the algebraic groups and algebraic varieties with their points over the algebraic closure $\overline{F}$ of $F$. We will view the $F$-points of the varieties as subsets invariant under the absolute Galois group $Gal(\overline{F}/F)$. As for the Lie algebra, we will use the notation $\fg$ to denote its $F$-points (as in the rest of the paper), and $\fg(\overline{F})$ will denote the points over the algebraic closure.
We start with several definitions and a theorem from \cite{KnoKro}.

\begin{defn}
An $F$-group is called \emph{elementary} if it is connected and all $F$-rational elements are semi-simple. An \emph{elementary radical} is the subgroup generated by $F$-rational unipotent elements. It is the smallest normal subgroup with elementary quotient. If $\bf H=LU$ is a Levi decomposition of an $F$-group $\bf H$ then ${\bf H}_{el}={\bf L}_{el}{\bf U}$, where ${\bf L}_{el}\subset {\bf L}$ is the product of all non-anisotropic simple factors.
\end{defn}
Note that a group over an algebraically closed field is elementary if and only if it is a torus.

\begin{defn}
Fix  a minimal $F$-parabolic subgroup ${\bf P}_0$ of $\bf G$. The $F$-adapted parabolic $\bf P$ of $\bf X$ that includes ${\bf P}_0$ is defined by
\begin{equation}
{\bf P}:=\{g\in \bfG \, \vert \, g{\bf P_0}x={\bf P_0} x \text{ for $x$ in a dense open subset of } {\bf X}\}
\end{equation}
We denote the $F$-points of $\bf P$ by $P$. We will call $P$ an \emph{adapted parabolic of $X$.}
\end{defn}
Most of the statements in \cite{KnoKro} deal with $F$-dense varieties. They are applicable to our case, since we assume $\bf X$ to be irreducible,  $X$ to be non-empty, and $F$ to be a local field. Under these assumptions, \cite[\S 1.A.2 and Proposition 2.6]{Pop} implies that $\bf X$ is $F$-dense, {\it i.e.} $X$ is Zariski dense in $\bf X$.

\begin{theorem}[{Local structure theorem, \cite[Corollary 4.6]{KnoKro}}]\label{thm:LS}
Let $\bf P=LU$ be a Levi decomposition of an adapted parabolic of $\bf X$. Then
there exists a smooth affine $\bf L$-subvariety ${\bf X}_{el}\subset {\bf X}$ such that
\begin{enumerate}[(i)]
\item ${\bf L}_{el}$ acts trivially on ${\bf X}_{el}$, all the $\bf L$-orbits on ${\bf X}_{el}$ are closed, and the categorical quotient ${\bf X}_{el}\to {\bf X}_{el}//{\bf L}$ is a locally trivial bundle in etale topology.
\item The natural morphism
\begin{equation}\label{=LST}
{\bf U}\times {\bf X}_{el}={\bf P}\times ^{\bf L} {\bf X}_{el}\to {\bf X}
\end{equation}
is an open embedding.
\end{enumerate}

\end{theorem}
\DimaE{
Recall that we denote by $\cM\subset \fg^*$ the closure of the image of the moment map of $X$.

\begin{cor}\label{cor:PG}
If $P=G$ then $\cM=\{0\}$.
\end{cor}}

It is easy to see from the definition, that any subgroup of an elementary group is reductive. Thus Theorem \ref{thm:LS} implies the following corollary.

\begin{cor}\label{cor:red}
There exists an open dense $P$-invariant subset $X^0$ of $X$ such that the stabilizer in $P$ of any point in $X^0$ is reductive.
\end{cor}

Theorem \ref{thm:LS} holds over any field of characteristic zero. It was first proven for algebraically closed fields in \cite{BLV}. We can therefore consider the adapted  parabolic of $\bf X$ over the algebraic closure $\overline{F}$ and compare it to $\bf P$.

\begin{proposition}[{\cite[Proposition 9.1]{KnoKro}}]\label{prop:PQ}
Let $\bf B\subset P_0$ be a Borel subgroup of $\bf G$, and let $\bf Q$ be the adapted parabolic subgroup of $\bf X$ that includes $\bf B$. Then ${\bf P}={\bf P}_0{\bf Q}$.
\end{proposition}

% \begin{cor}\label{cor:PQ}
% If $G$ is quasi-split then $\bf P$ is an adapted parabolic of $\bf X$.
% \end{cor}

Let $\bf Q^-=MU^-$ be a parabolic subgroup of $\bf G$ opposite to $\bf Q$ (where $\bf M=\bf Q\cap Q^-$), let $\bf M_s$ be the stabilizer in $\bf M$ of a point in ${\bf X}_{el}$ (where ${\bf X}_{el}$ comes from Theorem \ref{thm:LS} applied over $\overline{F}$), and let $\bf S$ be the preimage of $\bf M_s$ under the projection $\bf Q^-\onto M$.
Let us now consider the moment map over the algebraic closure $\mu:T^*{\bf X}\to \fg(\overline{F})^*$. Let $\fs$ and $\fq^-$ denote the Lie algebras of $\bf S$ and $\bf Q^-$.

\begin{thm}[{\cite[Appendix A]{PraSak} using \cite[Satz 5.4]{Kno}}]\label{thm:Kno}
$$ \overline{\mu(T^*{\bf X})}={\bfG}\cdot\fs^{\bot}$$
\end{thm}

\begin{cor}\label{cor:Geo}
We have
$$\cM\cap \cN\subset \overline{\mu(T^*{\bf X})}(F)\cap \cN= ({\bf G}\cdot(\fq^-)^{\bot})(F)$$ \end{cor}
\begin{proof}
Since $\bf M/M_s$ is elementary over $\overline{F}$, it is a torus. Thus the nilpotent elements of $\fs^{\bot}$ lie in $(\fq^{-})^{\bot}$. The corollary follows now from Theorem \ref{thm:Kno}.
%
% The moment map $T^*({\bf G/Q^{-}})\to \fg(\overline{F})^*$ is proper since it is the composition of the closed embedding $T^*(\bfG/{\bf Q^-})\into (\bfG/{\bf Q^-})\times \fg(\overline{F})^* \onto \fg(\overline{F})^* $ and the projection $(\bfG/{\bf Q^-})\times \fg(\overline{F})^* \onto \fg(\overline{F})^* $.
% Thus, its image  $\bf G \cdot(\fq^{-})^{\bot}$ is closed in $\fg(\overline{F})^*$.
\end{proof}

Note that the Killing form identifies $\bf (\fq^-)^{\bot}$ with $\fu^-$.

\section{Proof of Proposition \ref{prop:PrinSer} and Corollaries \ref{cor:main} and \ref{cor:sym} }\label{sec:PfCor}
% Let $P$ be the parabolic adapted to $X$. Let $\cO\subset X$ be as in the local structure theorem.
%
% We will need the following theorem, that follows from \cite[Theorem D]{GSS}.
%
% \begin{thm}\label{thm:AC}
% Let $\chi$ be a character of $P$ such that $\Sc^*(\cO)^{P,\chi}\neq 0$. Then $\Sc^*(X)^{P,\chi}\neq 0$.
% \end{thm}

\begin{proof}[Proof of Proposition \ref{prop:PrinSer}]
\eqref{PS:irr} is a well-known consequence of Bruhat theory, see \cite[\S 4]{KV} for the Archimedean case. The non-Archimedean case is proven  analogously.

\eqref{PS:WO} follows from Theorem \ref{thm:Twisted}. Indeed, we take $\zeta$ to be trivial and $\xi_m\circ\psi$ to be $\Delta_P^{-1/2}$. Then $p_{\fh}^{-1}(d\zeta)=\fp^{\bot}$. Since $G\cdot\fp^{\bot}$ is closed and lies in $\cN$, we obtain that  $$G\cdot p_{\fh}^{-1}(d\zeta)\cap \cN=G\cdot \fp^{\bot}=\overline{G\cdot Fp_{\fh}^{-1}(d\zeta)}\cap \cN.$$
By Theorem \ref{thm:Twisted} we obtain $G\cdot\fp^{\bot}=\WO(\Sc(G)_{P,\Delta_P^{-1/2}})=\WO(\pi)$.

For \eqref{PS:Quot} note that $P\cap H$ is the stabilizer of the coset $[1]\in G/H$, that lies in an open $P$-orbit. Thus, by Corollary \ref{cor:red}, $P\cap H$ is reductive, and thus $\Delta_{P\cap H}=1$.

\eqref{PS:Quot} follows now from \cite[Proposition D]{GSS}. More precisely, \cite[Proposition D]{GSS} is the case of trivial $\chi$, but the general case is proven in the same way.
\end{proof}

%\section{Proof of Corollaries \ref{cor:main} and \ref{cor:sym}}\label{sec:PfCor}

\begin{proof}[Proof of Corollary \ref{cor:main}]$\,$
\eqref{cormain:FinDim} Let $\pi\in \Irr(G)_H$. Present $G$ as a  quotient of $Z\times K\times M$ \DimaD{by a finite subgroup}, where $Z,K,M$ are reductive groups with $Z$ commutative, $K$ compact and $M$ generated by its unipotent elements. Then $\pi\in \Irr(Z\times K\times M)$, and we can assume $G=Z\times K\times M$.
Since $P=G$, Theorem \ref{thm:Twisted} and Corollary \ref{cor:PG} imply that $\WO(\widetilde{\pi})=\{0\}$. Thus \cite[Theorem 1.4]{GGS2} implies that $M$ acts locally finitely on $\pi$. Now, $Z$ acts on $\pi$ by scalars by the Schur-Dixmier lemma, and $\pi$ has a $K$-finite-vector. Since $\pi$ is irreducible, we obtain that it is finite-dimensional.\\
\eqref{cormain:ExFinDim} follows from Proposition \ref{prop:PrinSer}.\\
\eqref{cormain:ExGen} Let $\pi:=Ind_P^G1$. As explained in the proof of Proposition \ref{prop:PrinSer}(\ref{PS:irr},\ref{PS:WO}), $\pi$ is irreducible and generic. As before, $P\cap H$ is reductive and thus is unimodular. By \cite[Corollary C]{GSS}, $\pi\in \Irr(G)_{H,\chi}$. Again, \cite[Corollary C]{GSS} is the case of trivial $\chi$, but the general case is proven in the same way.\\
\eqref{GH:NoGen} If $G$ is not quasi-split then it has no generic representations. If $G$ is quasi-split then by \DimaD{Corollary \ref{cor:Irr}, Proposition \ref{prop:PQ} and Corollary \ref{cor:Geo} we have $\WO(\pi)\subset\overline{G\cdot \fh^{\bot}}\cap\cN=\cM\cap \cN\subset {\bfG\cdot\fp^{\bot}}$.
 If $P$ is not a Borel subgroup then $\bf \fp^{\bot}$ has no regular nilpotent elements, and thus neither has
$\WO(\pi)$.}
\end{proof}

\begin{lemma}\label{lem:uni}
If $PH$ is open in $G$ and $\Delta_H=1$ then  $\Delta_{P}|_{P\cap H}=1$.
\end{lemma}
\begin{proof}
Since $\Delta_H=1$, there exists a $G$-invariant measure on $G/H$. The restriction of this measure to $PH/H\cong P/(P\cap H)$ is $P$-invariant. Thus $\Delta_{P}|_{P\cap H}=\Delta_{P\cap H}$. But $\Delta_{P\cap H}=1$.
\end{proof}

\begin{proof}[Proof of Corollary \ref{cor:sym}]
Let $\theta$ be the involution of $G$ such that $H=G^{\theta}$.
Let $P$ be an adapted parabolic for $G/H$ such that $PH$ is open in $G$.
Since $H$ is reductive, $\Delta_H=1$. By Lemma \ref{lem:uni}, $\Delta_{P}|_{P\cap H}=1$. Thus, by Corollary \ref{cor:main}, the existence of generic $\pi\in \Irr(G)_H$ is equivalent to $P$ being a Borel, and the non-existence of infinite-dimensional $\pi\in \Irr(G)_H$ is equivalent to $P$ being $G$.
Now, $P=G$ if and only if $G/H$ is a torus, which in turn is equivalent to $G^H_{\R}$ being a quotient of the product of a compact group and a torus. This proves (ii). To prove (i),  we use the fact that $P$ is a Borel if and only if there exists a $\theta$-split Borel. The existence of a $\theta$-split Borel is equivalent to $G^H_{\R}$ being quasi-split by \cite{PraSak}.
\end{proof}
\DimaC{
\section{Proof of Corollary \ref{cor:global}}\label{sec:global}
Let $\pi\cong\otimes'_{\mu}\pi_{\mu}$ be the decomposition of $\pi$ into a restricted tensor product of irreducible local factors. Let $\eta$ be a non-zero $\bfH(\mathbb{A})$-equivariant functional on $\pi$. Then $\eta$ does not vanish on some pure tensor $\otimes w_{\mu}\in \pi\cong\otimes'_{\mu}\pi_{\mu}$. For any $\nu$, we
let $G:=\bfG(\mathbb{K}_{\nu})$ and $H:=\bfH(\mathbb{K}_{\nu})$ and
construct a non-zero $H$-invariant functional $\zeta_{\nu}$ on $\pi_{\nu}$ by $\zeta_{\nu}(w):=\eta(w\otimes \bigotimes_{\mu\neq \nu}w_{\mu})$.
By a similar argument, we have $\WO(\pi)\subset \WO(\pi_{\nu})$.} \DimaD{The existence of $\zeta_{\nu}$ implies by Corollary \ref{cor:Irr} that $\WO(\pi_{\nu})\subset \overline{G\cdot \fh^{\bot}(\mathbb{K}_{\nu})}$. Altogether we have $\WO(\pi)\subset \WO(\pi_{\nu})\subset \overline{G\cdot \fh^{\bot}(\mathbb{K}_{\nu})}$.
 \proofend
}

%
% \begin{proof}[Proof of  Corollary \ref{cor:main}]
% We have $\cM\cap \cN\subset G\fp^{\bot}$ by Theorem \ref{thm:Geo},$G\fp^{\bot}\subset \WO(\Sc(X))$ by Theorem \ref{thm:PrinSer}, and  $\WO(\Sc(X))\subset \cM\cap \cN$ by Theorem \ref{thm:WO}.
% \end{proof}
%

\end{document}